\documentclass[12pt,a4paper,reqno]{amsart}
\allowdisplaybreaks
\usepackage{amsmath}
\usepackage{amsfonts}
\usepackage{amssymb,amsthm,amsfonts,amsthm,latexsym,enumerate,url,cases}
\numberwithin{equation}{section}
\usepackage{mathrsfs}
\usepackage{hyperref}
\hypersetup{colorlinks=true,citecolor=blue,linkcolor=blue,urlcolor=blue}
\def\pmod #1{\ ({\rm{mod}}\ #1)}

\theoremstyle{plain}
\newtheorem{theorem}{Theorem}

\newtheorem{lemma}{Lemma}

\newtheorem{conjecture}{Conjecture}
\theoremstyle{definition}

\newtheorem{remark}{Remark}

\usepackage{etoolbox}
\makeatletter
\patchcmd{\@settitle}{\uppercasenonmath\@title}{}{}{}
\patchcmd{\@setauthors}{\MakeUppercase}{}{}{}
\patchcmd{\section}{\scshape}{}{}{}
\makeatother

\begin{document}

\title
[{Note on a conjecture of Ram\'{\i}rez Alfons\'{\i}n and Ska{\l}ba}]
{Note on a conjecture of Ram\'{\i}rez Alfons\'{\i}n and Ska{\l}ba}
\author
[T. Dai, \quad Y. Ding \quad H. Wang]
{Tianhan Dai, \quad Yuchen Ding$^\dagger$ \quad {\it and} \quad Hui Wang}

\address{(Tianhan Dai) School of Mathematical Sciences, Yangzhou University, Yangzhou 225002, People's Republic of China}
\email{3408606588@qq.com}

\address{(Yuchen Ding) School of Mathematical Sciences,  Yangzhou University, Yangzhou 225002, People's Republic of China}
\email{ycding@yzu.edu.cn}

\address{(Hui Wang) School of Mathematics, Shandong University, Jinan 250100, China, People's Republic of China}
\email{wh0315@mail.sdu.edu.cn}

\thanks{$^\dagger$ ORCID:  0000-0001-7016-309X \quad $\&$ \quad corresponding author}


\begin{abstract}
Let $2< a<b$ be two relatively prime integers and $g=ab-a-b$. It is proved that there exists at least one prime $p\le g$ of the form $p=ax+by~(x,y\in \mathbb{Z}_{\ge 0})$, which confirms a 2020 conjecture of Ram\'{\i}rez Alfons\'{\i}n and Ska{\l}ba.
\end{abstract}
\maketitle

\section{Introduction}
Let $a_1,a_2,\dots,a_k$ ($k\ge 2$) be a set of positive integers with $\gcd(a_1,a_2,\dots,a_k)=1$. The Diophantine Frobenius Problem (see, e.g. \cite{RA}) asks what is the largest integer $g_{a_1,a_2,\dots,a_k}$ which cannot be represented as
$$
a_1x_1+a_2x_2+\cdots+a_kx_k \quad \big(x_1,x_2,\dots,x_k\in \mathbb{Z}_{\ge 0}\big),
$$
where $\mathbb{Z}_{\ge 0}$ is the set of nonnegative integers.
The problem may seem at a glance not so hard, but the solution to it is actually extremely difficult. The closed form for $k=2$ was given by Sylvester \cite{Sylvester} who showed that $g_{a,b}=ab-a-b$. For $k=3$, closed forms involving particular cases were widely studied (see, e.g. \cite{RA}). Sylvester also obtained that for any $0\le s\le g_{a,b}$ exactly one of $s$ and $g_{a,b}-s$ could be expressed as the form
$
ax+by~ (x,y\in \mathbb{Z}_{\ge 0}).
$
Thus, exactly half of the integers in the interval $[0,g_{a,b}]$ have the desired expression. For comprehensive literature on the Diophantine Frobenius Problem, one can refer to the excellent monograph \cite{RA} of Ram\'{\i}rez Alfons\'{\i}n.

In a recent article, Ram\'{\i}rez Alfons\'{\i}n and Ska{\l}ba \cite{RS} considered the Diophantine Frobenius Problem in primes. Specifically, they were interested in the primes $p\le g_{a,b}$ which can be expressed in the form $p=ax+by$ for $x,y\in \mathbb{Z}_{\ge 0}$. Let us write $\pi_{a,b}$ for the number of such primes.
For any $\varepsilon>0$, Ram\'{\i}rez Alfons\'{\i}n and Ska{\l}ba proved that there is some constant $C(\varepsilon)>0$ so that
\begin{align}\label{eq1-1}
\pi_{a,b}>C(\varepsilon)\frac{g_{a,b}}{(\log g_{a,b})^{2+\varepsilon}}.
\end{align}
This means that $\pi_{a,b}>0$ for all sufficiently large $g_{a,b}$. A number of computer experiments then led them to the following conjecture \cite[Conjecture 2]{RS}.
\begin{conjecture}\label{conjecture1}
Let $2< a<b$ be two relatively prime integers. Then $\pi_{a,b}>0$.
\end{conjecture}
\begin{remark}\label{remark1}
In their original statement of \cite[Conjecture 2]{RS}, Ram\'{\i}rez Alfons\'{\i}n and Ska{\l}ba required $2\le a<b$. For $a=2$ and $b=3$, we have $g_{2,3}=2\cdot 3-2-3=1$. So, there would be no primes in the interval $[0,g_{2,3}]$. For $a=2$, $b>3$, it is obvious that $2=2\cdot1 +b\cdot0$ is a prime in the interval $[0,g_{2,b}]$. Therefore, without loss of generality we will only consider the case for $a>2$.
\end{remark}
On observing the antisymmetry property of the integers with the form  $cx+dy~(x,y\in \mathbb{Z}_{\geqslant0})$ discovered by Sylvester, Ram\'{\i}rez Alfons\'{\i}n and Ska{\l}ba also posed another conjecture \cite[Conjecture 3]{RS}.
\begin{conjecture}\label{conjecture2}
Let $2< a <b$ be two relatively prime integers, then
$$
\pi_{a,b}\sim\frac{\pi(g_{a,b})}{2} \quad (\text{as}~a\rightarrow\infty),
$$
where $\pi(t)$ is the number of primes up to $t$.
\end{conjecture}

Conjecture \ref{conjecture2} was investigated by Ding \cite{Ding}, who proved an almost all version of it. In a more recent article, Ding, Zhai and Zhao \cite{DingZhaiZhao} succeeded in proving Conjecture \ref{conjecture2}. The result of Ding, Zhai and Zhao was later generalized to the $\ell$-Frobenius number by Ding and Komatsu \cite{DK}. Very recently, Huang and Zhu \cite{HuangZhu} further extended the result of Ding, Zhai and Zhao to the distribution of prime powers.

In this note, we turn to Conjecture \ref{conjecture1} and confirm it here.
\begin{theorem}\label{thm1}
Let $2< a<b$ be two relatively prime integers. Then we have $\pi_{a,b}>0$.
\end{theorem}

\section{Proof of Theorem \ref{thm1}}
For simplicity, we use $g$ to denote $g_{a,b}$. We will follow the proof of Ram\'{\i}rez Alfons\'{\i}n and Ska{\l}ba \cite[Theorem 1]{RS} with explicit calculations.
Since the constant $C(\varepsilon)$ in Eq. (\ref{eq1-1}) is not specific in their argument, we need an effective version of the Siegel--Walfisz theorem as well as an effective version of the prime number theorem.

Let $\pi(x;q,a)$ be the number of primes $p$ not exceeding $x$ so that $p\equiv a\pmod{q}$ and
$$
{\rm Li}(x)=\int_2^{x}\frac{1}{\log t}\mathrm{dt}.
$$
\begin{lemma}\cite[Theorem 1.3]{Bennett}.\label{lem1}
Let $q\ge 3$ be an integer and let $m$ be an integer that is coprime to $q$. There exist explicit constants $c_\pi(q)$ and $x_\pi(q)$ such that
$$
\left|\pi(x;q,m)-\frac{{\rm Li}(x)}{\varphi(q)}\right|<c_\pi(q)\frac{x}{(\log x)^2} \quad \text{for~all}~x\ge x_\pi(q).
$$
Moreover, $c_\pi(q)\le c_0(q)$ and $x_\pi(q)\le x_0(q)$, where
\begin{align*}
c_0(q)=
\begin{cases}
\frac{1}{840}, & \text{if~}3\le q\le 10^4,\\
\frac{1}{160}, & \text{if~}q>10^4,
\end{cases}
\end{align*}
and
\begin{align}\label{remark-1}
x_0(q)=
\begin{cases}
8\cdot 10^9, & \text{if~}3\le q\le 10^{4},\\
\exp\Big(0.03\sqrt{q}(\log q)^3\Big), & \text{if~}q>10^{4}.
\end{cases}
\end{align}
\end{lemma}

\begin{remark}
In the original statement of \cite[Theorem 1.3]{Bennett}, the result involving (\ref{remark-1}) is 
\begin{align}\label{remark-2}
x_0(q)=
\begin{cases}
8\cdot 10^9, & \text{if~}3\le q\le 10^{5},\\
\exp\Big(0.03\sqrt{q}(\log q)^3\Big), & \text{if~}q>10^{5}.
\end{cases}
\end{align}
Actually, (\ref{remark-1}) follows from (\ref{remark-2}) by noting that
$$
\exp\Big(0.03\sqrt{q}(\log q)^3\Big)>8\cdot 10^9
$$
for any $10^{4}\le q\le 10^5$.
\end{remark}

\begin{lemma}\cite[Theorem 1]{Panaitopol}.\label{lem3}
For $x\ge 59$ we have
$$
\frac{x}{\log x-1+(\log x)^{-0.5}}<\pi(x)<\frac{x}{\log x-1-(\log x)^{-0.5}}.
$$
\end{lemma}

\begin{lemma}\cite[Corollary 1.7]{Bennett}.\label{lem2}
Let $m$ and $q$ be integers with $1\le q\le 10^5$ and $\gcd(m,q)=1$. If $x\ge 10^6$, then
$$
\left|\pi(x;q,m)-\frac{{\rm Li}(x)}{\varphi(q)}\right|<0.027\frac{x}{(\log x)^2}.
$$
\end{lemma}

\begin{lemma}\cite[Theorem 1.9]{Bennett}.\label{newlemma}
Let $q$ and $m$ be integers with $100\le q\le 10^4$ and $\gcd(m,q)=1$. Then for $x\le 10^{11}$ we have
$$
\max_{1\le y\le x}\left|\pi(y;q,m)-\frac{{\rm Li}(y)}{\varphi(q)}\right|<2.734\frac{\sqrt{x}}{\log x}.
$$
\end{lemma}

Before the proof of our theorem, we first note that
$$
(a-1)a\le (a-1)(b-1)=g+1 \quad \text{and} \quad b^2>(a-1)(b-1)>g
$$
which clearly means that $b>\sqrt{g}$ and further $a<\sqrt{g}+1$.

We now proceed to the proof of Theorem \ref{thm1}.

\begin{proof}[Proof of Theorem \ref{thm1}]
The proof will be separated into several cases.

{\bf Case I.} $10^4<a\le 4.2\log b$. In this case, we have
\begin{align}\label{eq2-1}
g>b\ge \exp\left(\frac{a}{4.2}\right)>\exp\Big(0.03\sqrt{a}(\log a)^3\Big).
\end{align}
We are going to count the primes $p\le g$ with $p\equiv b\pmod{a}$.
We note that if a prime $p$ satisfies the congruence $p\equiv b\pmod{a}$ with $p>b$, then
$$
p=b+\frac{p-b}{a}a
$$
is a prime satisfying our requirements. Thus, by (\ref{eq2-1}) and Lemma \ref{lem1} we get
\begin{align}\label{eq2-2}
\pi_{a,b}\ge \pi(g;a,b)-\pi(b;a,b)\ge \frac{{\rm Li}(g)}{\varphi(a)}-\frac{{\rm Li}(b)}{\varphi(a)}+2R(g),
\end{align}
where $|R(g)|<\frac{1}{160}g/(\log g)^2$.  Moreover, from (\ref{eq2-2}) we have
\begin{align*}
\pi_{a,b}&\ge \frac{{\rm Li}(g)}{\varphi(a)}-\frac{{\rm Li}(b)}{\varphi(a)}-\frac{g}{80(\log g)^2}\\
&>\frac{1}{\varphi(a)\log g}\int_b^{g}dt-\frac{g}{80(\log g)^2}\\
&\ge \frac{g-b}{4.2(\log g)^2}-\frac{g}{80(\log g)^2}>0
\end{align*}
since $10^4<a\le 4.2(\log b)<4.2\log g$, $80(g-b)-4.2g>0$ where $b=(g+a)/(a-1)$.

{\bf Case II.} $a>10^4$ and $155(\log b)\le a<\sqrt{g}+1$. Let $c=g/\log g$. For any $0\le y\le \left\lfloor c/b \right\rfloor$, the number of $x\in \mathbb{Z}_{\ge 0}$ so that
$
ax+by\le c
$
is exactly
$$
\left\lfloor\frac{c-by}{a}\right\rfloor+1.
$$
Recall that for $ax_1+by_1, ax_2+by_2< g$ with $ax_1+by_1=ax_2+by_2$ we must have $x_1=x_2$ and $y_1=y_2$ since $(a,b)=1, x_1,x_2<b$, and $y_1,y_2<a$.
Thus, the number of integers $n\in [0,c]$ with the form
$
ax+by~ (x,y\in \mathbb{Z}_{\ge 0})
$
is
\begin{align}\label{eq2-3}
\sum_{y=0}^{\left\lfloor c/b \right\rfloor}\left(\left\lfloor\frac{c-by}{a}\right\rfloor+1\right).
\end{align}
Now, by the result of Sylvester (see e.g. \cite{RS}) mentioned in the introduction, we know that for any $0\le s\le g$ exactly one of $s$ and $g-s$ could be expressed as the form
$
ax+by~ (x,y\in \mathbb{Z}_{\ge 0}).
$
Thus, the total number of integers in the interval $[g-c,g]$ without the form $
ax+by~ (x,y\in \mathbb{Z}_{\ge 0})
$
equals (\ref{eq2-3}). 
Hence, we have
\begin{align}\label{eq2-4}
\pi_{a,b}\ge \pi(g)-\pi(g-c)-\sum_{y=0}^{\left\lfloor c/b \right\rfloor}\left(\left\lfloor\frac{c-by}{a}\right\rfloor+1\right),
\end{align}
since the right-hand side of (\ref{eq2-4}) is exactly the number of primes of the form $ax+by$ in $(g-c,g]$. It is clear that
\begin{align}\label{eq2-5}
\sum_{y=0}^{\left\lfloor c/b \right\rfloor}\left(\left\lfloor\frac{c-by}{a}\right\rfloor+1\right)
&\le \left\lfloor \frac{c}{a}\right\rfloor+1+\sum_{y=1}^{\left\lfloor c/b \right\rfloor}\left(\frac{c-by}{a}+1\right)\nonumber\\
&\le \left\lfloor \frac{c}{a}\right\rfloor+1+\frac{1}{2}\left(\frac{c-b}{a}+1+\frac{b}{a}+1\right)\left\lfloor \frac{c}{b} \right\rfloor\nonumber\\
&\le \frac{c}{a}+1+\left(\frac{c}{2a}+1\right)\frac{c}{b}\nonumber\\
&=\frac{c^2}{2ab}+\frac{c}{a}+\frac{c}{b}+1.
\end{align}
By (\ref{eq2-4}), (\ref{eq2-5}) and Lemma \ref{lem3}, we have
\begin{align}\label{eq2-6}
\pi_{a,b}\ge\frac{g}{\log g-1+(\log g)^{-0.5}}~-~&\frac{g-g/\log g}{\log (g-g/\log g)-1-\log^{-0.5} (g-g/\log g)}\nonumber\\
&-\frac{g}{2(\log g)^2}-\frac{g}{155(\log b)(\log g)}-\frac{\sqrt{g}}{\log g}-1
>0
\end{align}
providing that $g>5\cdot 10^8$ via computer verification.\footnote{We use the software  to do this. For the verifications, we firstly use the inequality $b>\sqrt{g}$. And then the middle equation of (\ref{eq2-6}) turns to be a function of $g$ which we will denote it by $f(g)$, say. By computer, we know $f(g)\ge f(5\cdot 10^8)>0$ for $g\in \big[5\cdot10^8,10^{18}\big]$. Actually, the minimum value is 566.0054846 (obtained by taking $g=5\cdot 10^8$). For $g>10^{18}$, we use the Taylor expansion to get
$$
\frac{g}{\log g-1+(\log g)^{-0.5}}>\frac{g}{\log g}+\frac{g}{(\log g)^2}-\frac{g}{(\log g)^{2.5}}
$$
and use the Taylor expansion as well as the inequality $t/(1+t)\le \log (1+t)\le t,~(t>-1)$ to get
$$
\frac{g-g/\log g}{\log (g-g/\log g)-1-\log^{-0.5} (g-g/\log g)}\!<\!\frac{g}{\log g}\big(1-\frac{1}{\log g}\big)\big(1+\frac{5}{4\log g}\big)\!=\!\frac{g}{\log g}+\frac{g}{4(\log g)^2}-\frac{5}{4(\log g)^3}.
$$
Then for $g>10^{18}$ we have 
$$
f(g)>\frac{147g}{620(\log g)^2}-\frac{g}{(\log g)^{2.5}}+\frac{5}{4(\log g)^3}-\frac{\sqrt{g}}{\log g}-1>0.
$$} 

{\bf Case III.} $a>10^4$ and $4.2(\log b)<a<155(\log b)$. In this case, it follows clearly that
$g>b>\exp\left(10^4/155\right)$.
By the same arguments (with $155$ replaced by $4.2$) leading to (\ref{eq2-6}) in {\bf Case II}, we have for $g>b>\exp\left(10^4/155\right)$
\begin{align*}
\pi_{a,b}\ge\frac{g}{\log g-1+(\log g)^{-0.5}}~-~&\frac{g-g/\log g}{\log (g-g/\log g)-1-\log^{-0.5} (g-g/\log g)}\nonumber\\
&-\frac{g}{2(\log g)^2}-\frac{g}{4.2(\log b)(\log g)}-\frac{\sqrt{g}}{\log g}-1>0
\end{align*}
with the help of computers.\footnote{The verifications are similar to {\bf Case II} (i.e., footnote 1 in page 4).}

Thus, for $a>10^4$, it remains to verify the case of $g\le 5\cdot 10^8$ with the constraint $a\ge 155(\log b)$ from {\bf Cases I, II} and {\bf III}.

We next treat the cases of $3\le a\le 10^4$. For a prime $a$ (resp. $b$), it is clear that
$$
a=a\cdot1+b\cdot0 \quad (\text{resp.~} b=a\cdot0+b\cdot1)
$$
is a prime satisfying our requirement. Thus, we will only consider the case that both $a$ and $b$ are not primes (hence $a\ge 4$).

{\bf Case IV.} $4\le a\le 10^4$, $a\le 210\log b<210\log g$ and $b>8\cdot 10^9$.  Then
$$
g=ab-a-b=(a-1)(b-1)-1\ge 3 \cdot8\cdot 10^9-1>8\cdot 10^9
$$ and $b=(g+a)/(a-1)\le g/2$ for $a\ge 4$. Then by Lemma \ref{lem1} we have
\begin{align*}
\pi_{a,b}&\ge \pi(g;a,b)-\pi(b;a,b)\ge \frac{{\rm Li}(g)}{\varphi(a)}-\frac{{\rm Li}(b)}{\varphi(a)}+2R(g),
\end{align*}
where $|R(g)|<\frac{1}{840}g/(\log g)^2$. Thus,
\begin{align*}
\pi_{a,b}&\ge \frac{{\rm Li}(g)}{\varphi(a)}-\frac{{\rm Li}(b)}{\varphi(a)}-\frac{g}{420(\log g)^2}\\
&>\frac{1}{\varphi(a)\log g}\int_b^{g}dt-\frac{g}{420(\log g)^2}\\
&\ge \frac{g-b}{\varphi(a)(\log g)}-\frac{g}{420(\log g)^2}\\
&> \frac{g}{420(\log g)^2}-\frac{g}{420(\log g)^2}=0
\end{align*}
since $a\le 210\log b<210\log g$ and $b\le g/2$.

{\bf Case V.} $4\le a\le 10^4$, $210\log b<a<\sqrt{g}+1$ and $b>8\cdot 10^9$. Then $g\ge2b>16\cdot 10^9$.
By the same arguments (with $155$ replaced by $210$) leading to (\ref{eq2-6}) in {\bf Case II}, we have
\begin{align*}
\pi_{a,b}\ge\frac{g}{\log g-1+(\log g)^{-0.5}}~-~&\frac{g-g/\log g}{\log (g-g/\log g)-1-\log^{-0.5} (g-g/\log g)}\nonumber\\
&-\frac{g}{2(\log g)^2}-\frac{g}{210(\log b)(\log g)}-\frac{\sqrt{g}}{\log g}-1>0
\end{align*}
with the help of computers.\footnote{The verifications are similar to {\bf Case II} (i.e., footnote 1 in page 4).}

From {\bf Cases IV and V}, for $4\le a\le 10^4$ we only need to verify $b\le 8\cdot 10^9$. We will try to decrease the domain of verifications by additional efforts.

{\bf Case VI.} $4\le a\le 808$ and $10^6\le b\le 8\cdot 10^9$. Then
$$
g=ab-a-b=(a-1)(b-1)-1\ge 3\cdot \left(10^6-1\right)-1>10^6.
$$
By Lemma \ref{lem2} we have
\begin{align*}
\pi_{a,b}&\ge \pi(g;a,b)-\pi(b;a,b)\ge \frac{{\rm Li}(g)}{\varphi(a)}-\frac{{\rm Li}(b)}{\varphi(a)}+R(g)+R(b),
\end{align*}
where $|R(g)|<\frac{0.027g}{(\log g)^2}$ and $|R(b)|<\frac{0.027b}{(\log b)^2}$. Thus, with the help of computers we know that for $4\le a\le 808$,\footnote{We use the `{\it Minimize}' command of software to see that the minimum value is 42.5025 (obtained by taking $a=797$ and $b=10^6$).}
\begin{align*}
\pi_{a,b}&\ge \frac{{\rm Li}(g)}{\varphi(a)}-\frac{{\rm Li}(b)}{\varphi(a)}-\frac{0.027g}{(\log g)^2}-\frac{0.027b}{(\log b)^2}>\frac{1}{\varphi(a)}\int_b^{g}\frac{1}{\log t}dt-\frac{0.027g}{(\log g)^2}-\frac{0.027b}{(\log b)^2}>0.
\end{align*}

{\bf Conclusion.} From all above cases and discussions, it remains to show $\pi_{a,b}>0$ for

(i) $a>10^4$ and $g\le 5\cdot 10^8$ with the constraints $a\ge 155\log b$ and $a,b$ not primes.

(ii) $808<a\le 10^4$ and $10^6\le b\le 8\cdot 10^9$ with $a,b$ not primes.

(iii) $a\le 10^4$ and $b<10^6$ with $a,b$ not primes.

Verifications of (i) and (iii) could be done by the computer.\footnote{We use the software of computers to do these verifications. Since the ranges of $a$ and $b$ are finite, we use the computer to travel all over the situations. For a given pair $a$ and $b$ satisfying the restrictions, we just search for the the first prime $p<g$ of the form $ax+b\cdot1$ which helps the computer to check our purpose more quickly and efficiently. Luckily, all the situations contain a prime $p$ of the form $ax+b\cdot1$. We could verify (i) within several minutes. It took two notebook computers around six days to accomplish the verifications of (ii).}
It is, however, still a heavy work for the computer to verify (ii). We will prove (ii) by theoretical arguments.

{\bf Case VII.} $808<a<10^4$, $\frac{8\cdot 10^9}{a-1}+1<b\le 8\cdot 10^9$. Then 
$$
g=ab-a-b=(a-1)(b-1)-1>8\cdot 10^9-1.
$$
Hence $g\ge 8\cdot 10^9$ since $g$ is an integer. Moreover,
$
g<ab<8\cdot 10^{13}.
$
Trivially, we have
$$
\pi(b;a,b)\le \frac{b}{a}+1.
$$
Thus, by Lemma \ref{lem1} we have
\begin{align*}
\pi_{a,b}&\ge \pi(g;a,b)-\pi(b;a,b)\ge \frac{{\rm Li}(g)}{\varphi(a)}-\left(\frac{b}{a}+1\right)+R(g),
\end{align*}
where $|R(g)|<\frac{1}{840}g/(\log g)^2$. Hence,  by trivial estimates we get
\begin{align*}
\pi_{a,b}&\ge \frac{{\rm Li}(g)}{\varphi(a)}-\left(\frac{b}{a}+1\right)-\frac{g}{840(\log g)^2}>0
\end{align*}
since $b=(g+a)/(a-1)<g/808+809/808$ and $8\cdot 10^9\le g<8\cdot 10^{13}$.

{\bf Case VIII.} $808<a\le 10^4$, $10^6\le b\le 10^7$. Since $8\cdot 10^8<g<ab<10^{11}$, by Lemma \ref{newlemma} with $x=g$ we have\footnote{We use the `{\it Minimize}' command of software to see that the minimum value is 21647.8 (obtained by taking $a=9973$ and $b=10^6$).}
\begin{align*}
\pi_{a,b}&\ge \pi(g;a,b)-\pi(b;a,b)
\ge \frac{1}{\varphi(a)}\int_{b}^{g}\frac{1}{\log t}dt-\frac{5.5\sqrt{g}}{\log g}>0.
\end{align*}

We proved (ii) by {\bf Cases VII} and {\bf VIII} since $\frac{8\cdot 10^9}{a-1}+1<10^7$ for $a>808$.
\end{proof}

\section*{Acknowledgments}
We thank the anonymous referees for their very helpful comments. We also thank Honghu Liu for his careful reading the manuscript. 

The verifications by computers were partly carried over
separately by Zhidan Feng, Guilin Li, Deyong L\"u, Jianxing Song, Rui-Jing Wang, Hao Zhang and Guang-Liang Zhou. We deeply thank them for their generous help and efforts on our article.
We also thank Yu Zhang and Tengyou Zhu for their interests of this article.

H. Wang would like to thank Professors Xiumin Ren and Qingfeng Sun for their help and encouragement, the Alfr\'ed R\'enyi Institute of Mathematics for proving a great working environment and the China Scholarship Council for its financial support. 

The second named author is supported by National Natural Science Foundation of China  (Grant No. 12201544), Natural Science Foundation of Jiangsu Province, China (Grant No. BK20210784), China Postdoctoral Science Foundation (Grant No. 2022M710121).
The third named author is supported by the National Natural Science Foundation of China (Grant Nos.
12471005 and 12031008) and
the Natural Science Foundation of Shandong Province (Grant No. ZR2023MA003).


\begin{thebibliography}{KMP}
\bibitem{Bennett}
M. A. Bennett, G. Martin, K. O'Bryant, A. Rechnitzer, {\it Explicit bounds for primes in arithmetic progressions,} Illinois J. Math. {\bf 62} (2018), 427--532.

\bibitem{Ding}
Y. Ding, {\it On a conjecture of Ram\'{\i}rez Alfons\'{\i}n and Ska{\l}ba,} J. Number Theory {\bf245} (2023), 292--302.

\bibitem{DK}
Y. Ding, T. Komats, {\em On a conjecture of Ram\'{\i}rez Alfons\'{\i}n and Ska{\l}ba III,} to appear in Integers, arXiv:2311.03997 (2023).

\bibitem{DingZhaiZhao}
Y. Ding, W. Zhai, L. Zhao, {\em On a conjecture of Ram\'{\i}rez Alfons\'{\i}n and Ska{\l}ba II,} to appear in J. Th\'eor. Nombres Bordeaux, arXiv:2309.09796v2 (2023).

\bibitem{HuangZhu}
E. Huang, T. Zhu, {\em The distribution of powers of primes related to the Frobenius problem}, Lith. Math. J. {\bf 65} (2025), 67--82.

\bibitem{Panaitopol}
L. Panaitopol, {\em Inequalities concerning the function $\pi(x)$: applications,} Acta Arith. {\bf 94} (2000), 373--381.

\bibitem{RA}
J. L. Ram\'{\i}rez Alfons\'{\i}n, {\em The Diophantine Frobenius Problem, Oxford Lecture Series in Mathematics and its Applications,} vol. 30, Oxford University Press, 2005.

\bibitem{RS}
J. L. Ram\'{\i}rez Alfons\'{\i}n, M. Ska{\l}ba, {\it Primes in numerical semigroups,} C. R. Math. Acad. Sci. Paris {\bf358} (2020), 1001--1004.


\bibitem{Sylvester}
J. J. Sylvester, {\it On subvariants, i.e. Semi--Invariants to Binary Quantics of an Unlimited Order,}  Amer. J. Math. {\bf105} (1882), 79--136.


\end{thebibliography}
\end{document}